\documentclass[12pt, twoside]{article}
\usepackage[ansinew]{inputenc}
\usepackage{amsmath,amsthm,amssymb}
\usepackage{times}
\usepackage{enumerate}

\def\titlerunning#1{\gdef\titrun{#1}}
\makeatletter
\def\author#1{\gdef\autrun{\def\and{\unskip, }#1}\gdef\@author{#1}}
\def\address#1{{\def\and{\\\hspace*{18pt}}\renewcommand{\thefootnote}{}%
\footnote {#1}}%
\markboth{\autrun}{\titrun}}
\makeatother
\def\email#1{e-mail: #1}
\def\subjclass#1{{\renewcommand{\thefootnote}{}%
\footnote{\emph{Mathematics Subject Classification (2010):} #1}}}
\def\keywords#1{\par\medskip
\noindent\textbf{Keywords.} #1}

\newtheorem{thm}{Theorem}[section]

\newtheorem{lem}[thm]{Lemma}
\newtheorem{prop}[thm]{Proposition}
\newtheorem{prob}[thm]{Problem}

\theoremstyle{definition}

\newtheorem{rem}[thm]{Remark}

\newtheorem{con}[thm]{Conjecture}

\numberwithin{equation}{section}

\def\R{\mathbb{R}}
\def\N{\mathbb{N}}
\def\Z{\mathbb{Z}}
\def\Q{\mathbb{Q}}

\begin{document}


\baselineskip=17pt


\titlerunning{Extended Lagrange's four-square theorem}

\title{Extended Lagrange's four-square theorem}

\author{Jesús Lacalle
\and
Laura N. Gatti}

\date{}

\maketitle

\address{J. Lacalle: Dep. de Matemática Aplicada a las Tecnologías de la Información y las Comunicaciones, ETSI de Sistemas Informáticos, Universidad Politécnica de Madrid, C/ Alan Turing s/n, 28031, Madrid, Spain; \email{jlacalle@etsisi.upm.es}
\and
L.N. Gatti: Dep. de Matemática Aplicada a las Tecnologías de la Información y las Comunicaciones, ETSI de Sistemas Informáticos, Universidad Politécnica de Madrid, C/ Alan Turing s/n, 28031, Madrid, Spain; \email{ln.gatti@alumnos.upm.es}}



\begin{abstract}

Lagrange's four-square theorem states that every natural number $n$ can be represented as the sum of four integer squares: $n=x_1^2+x_2^2+x_3^2+x_4^2$. Ramanujan generalized Lagrange's result by providing, up to equivalence, all $54$ quadratic forms $ax_1^2+bx_2^2+cx_3^2+dx_4^2$ that represent all positive integers. In this article, we prove the following extension of Lagrange's theorem: given a prime number $p$ and $v_1\in\Z^4$, $\dots$, $v_k\in\Z^4$, $1\leq k\leq 3$, such that $\|v_i\|^2=p$ for all $1\leq i\leq k$ and $\langle v_i|v_j\rangle=0$ for all $1\leq i<j\leq k$, then there exists $v=(x_1,x_2,x_3,x_4)\in\Z^4$ such that $\langle v_i|v\rangle=0$ for all $1\leq i\leq k$ and
$$
\|v\|^2=x_1^2+x_2^2+x_3^2+x_4^2=p
$$
This means that, in $\Z^4$, any system of orthogonal vectors of norm $p$ can be completed to a base. We conjecture that the result holds for every norm $p\geq 1$. The problem comes up from the study of a discrete quantum computing model in which the qubits have Gaussian integers as coordinates, except for a normalization factor $\sqrt{2^{-k}}$.

\end{abstract}

\section{Introduction}

Long before Lagrange proved his theorem, Diophantus had asked whether every positive integer could be represented as the sum of four perfect squares greater than or equal to zero. This question later became known as Bachet's conjecture, after the 1621 translation of Diophantus by Bachet. In parallel, Fermat proposed the problem of representing every positive integer as a sum of at most $n$ $n-$gonal numbers. Lagrange~\cite{Lag} proved the square case of the Fermat polygonal number theorem in 1770, also solving Bachet's conjecture. Gauss~\cite{Gau} proved the triangular case in 1796 and the full polygonal number theorem was not solved until it was finally proven by Cauchy in 1813. Later, in 1834, Jacobi discovered a simple formula for the number of representations of an integer as the sum of four integer squares.

The same year in which Lagrange proved his theorem, Waring asked whether each natural number $k$ has an associated positive integer $s$ such that every natural number is the sum of at most $s$ natural numbers to the power of $k$. For example, every natural number is the sum of at most $4$ squares, $9$ cubes, or $19$ fourth powers. The affirmative answer to the Waring's problem, known as the Hilbert–Waring theorem, was provided by Hilbert in 1909.

A possible generalization of Lagrange's problem is the following: given natural numbers $a$, $b$, $c$ and $d$, can we solve
$$
n=ax_1^2+bx_2^2+cx_3^2+dx_4^2
$$
for all positive integers $n$ in integers $x_1$, $x_2$, $x_3$ and $x_4$? Lagrange's four-square theorem answered in the positive the case $a=b=c=d=1$ and the general solution was given by Ramanujan~\cite{Ram}. He proved that if we assume, without loss of generality, that $a\leq b\leq c\leq d$ then there are exactly 54 possible choices for $a$, $b$, $c$ and $d$ such that the problem is solvable in integers $x_1$, $x_2$, $x_3$ and $x_4$ for all $n\in\N$.

Another possible generalization, due to Mordel~\cite{Mor}, tries to represent positive definite integral binary quadratic forms instead of positive integers. He proved that the quadratic form $x^2+y^2+z^2+u^2+v^2$ represents all positive definite integral binary quadratic forms.

Sun~\cite{Sun} has proposed some refinements of the Lagrange's theorem such as, for example, the following:
$n\in\N$ can be written as $x^2+y^2+z^2+w^2$ with $x,y,z,w\in\Z$ such that $x+y+z$ (or $x+2y$, or $x+y+2z$) is a square (or a cube).

The extension of the Lagrange's four-square theorem proposed in this article comes up from the study of the model of discrete quantum computation introduced by the authors~\cite{GL1}. In this model, the discrete quantum states (qubits) have Gaussian integers as coordinates, except for a normalization factor $\sqrt{2^{-k}}$. The model is constructed from two elementary quantum gates, $H$ and $G$. The Hadamard gate $H$ is one of the most relevant quantum gates that allows superposition, and therefore entanglement and parallelism.

The other gate, $G$, is a three qubit gate in which the first two are control qubits, while the third is the target. If the control qubits are in state $|1\rangle$ then the gate $V$ is applied to the third qubit.
$$
H=
\frac{1}{\sqrt{2}}
\left(\begin{array}{cc}
1 & 1 \\
1 & -1
\end{array}\right)
\quad\quad\quad
V=
\left(\begin{array}{cc}
1 & 0 \\
0 & i
\end{array}\right)
$$

These quantum gates allow the construction of all discrete states (states with integer real and imaginary parts, i.e. Gaussian integers, as coordinates). It is because of this fact that the authors call the second gate $G$ (for Gauss).

The model was designed to generate all discrete quantum states from the computational base. For this reason the proof of this fact was relatively simple. The defined discrete quantum gates in the model have discrete quantum states as columns (and as rows). As a matter of fact, the authors did not expect that the elementary quantum gates $H$ and $G$ could generate all discrete quantum gates, because this means simultaneously generating as many discrete quantum states as gate columns. But, surprisingly, this could be done and indicated to the authors that it might be true that an orthonormal system of discrete quantum states can always be completed to a base. In this article we include the simplest version of this problem, which was already presented as a conjecture at a conference by the authors~\cite{GL2}.

The outline of the article is as follows: In section 2 we set up notations and discuss some basic properties. In section 3 we prove the main result. Finally, in section 4 we expose several generalizations and conjectures related to the proposed problem.

\section{Notations and basic properties}

We consider $\Z^4$ as a part of the vector space $\R^4$ provided with the inner product $\langle v|w\rangle=x_1y_1+x_2y_2+x_3y_3+x_4y_4$, where $v=(x_1,x_2,x_3,x_4)$ and $w=(y_1,y_2,y_3,y_4)$ are vectors of $\R^4$, and with the canonical base $\{e_1,\dots,e_4\}$.

Given a set of linearly independent vectors $v_1,\dots,v_k\in\R^4$, they generate the \emph{lattice $\Lambda=\{\,b_1v_1+\cdots+b_kv_k\ |\ b_1,\dots,b_k\in\Z\,\}$}~\cite{Ca} and constitute a \emph{base of $\Lambda$}, $B$. So the \emph{dimension of $\Lambda$} will be $k$. From now on we will only consider bases whose vectors belong to $\Z^4$, i.e. $\Lambda$ will always be an \emph{integral lattice}.

Given a point $v\in\Lambda$, described by its coordinates in $B$, $v=(b_i)_B$, the number $N(v)=\|v\|^2=\langle v|v\rangle$ is called the \emph{norm of $v$} and can be calculated by the expression $N(v)=b^tGb$, where $G$ is the \emph{Gram matrix} of the vectors of $B$. The determinant of $G$, $\text{det}(G)$, is an invariant of $\Lambda$ whose square root is denoted by $\text{det}(\Lambda)$. So $\text{det}(\Lambda)=\sqrt{\text{det}(G)}$ and, geometrically, it is interpreted as the volume of the fundamental parallelepiped of $\Lambda$. The matrix $G$ is symmetric and positive definite and is associated to a quadratic form that collects the main properties of $\Lambda$.

Let us consider the \emph{coordinate matrix $V$}, formed by the vectors of the base $B$ of $\Lambda$ placed by rows. If $V$ is a square matrix, we can compute the determinant of $\Lambda$ from $V$, $\text{det}(\Lambda)=|\text{det}(V)|$, and it holds that $\text{det}^2(V)=\text{det}(G)$.

Given a set of vectors $v_1,\dots,v_k\in\Z^4$ such that $N(v_i)=p$ for all $1\leq i\leq k$ and $\langle v_i|v_j\rangle=0$ for all $1\leq i<j\leq k$, we will say that $S=\{\,v_1,\dots,v_k\,\}$ is a \emph{$p-$orthonormal system} and, if $k=4$, that $S$ is a \emph{$p-$orthonormal base}. The \emph{support of $S$} is $supp(S)=\{\,i\ |\ \exists j\ \text{such that the $i-$coordinate of}\ v_j\neq 0\,\}$.

However, we are not interested in $\Lambda$, but rather in its \emph{orthogonal lattice}
$$
\Lambda^{\bot}=\{\,v\in\Z^4\ |\ \langle v_i|v\rangle=0\ \text{for all}\ 1\leq i\leq k\,\}
$$
The resolution method of systems of linear Diophantine equations~\cite{CC} computes a base of $\Lambda^{\bot}$ with $4-k$ vectors. Then the dimension of $\Lambda^\bot$ will be $k^\bot=4-k$. In order to do this we have to solve the linear system $VX=0$, computing the \emph{Smith normal form}~\cite{Smi} of $V$ and its \emph{invariant factors $\alpha_1,\dots,\alpha_k$}:
$$
L\,VR =
\left(\begin{array}{cccc}
\alpha_1 &        & \\
  & \ddots & \\
  &        & \alpha_k
\end{array}\right) = N
\quad\text{such that}\quad
\begin{array}{l}
L\in GL_k(\Z) \\
R\in GL_4(\Z) \\
0<\alpha_1,\cdots,\alpha_k \\
\alpha_1 | \alpha_2,\,\dots\,\alpha_{k-1} | \alpha_k
\end{array}
$$

\begin{lem}\label{LemaBaseLatticeOrtogonal}
Given a number $p\geq 1$ and a $p-$orthonormal system $S=\{\,v_1,\dots,v_k\,\}$, $1\leq k\leq 3$, with associated lattice $\Lambda$, then the last $4-k$ columns of the matrix $R$, in the Smith normal form of $V$, constitute a base of $\Lambda^\bot$.
\end{lem}
\begin{proof}
It holds that $\,VX=0\,\Leftrightarrow\,L\,VR\,R^{-1}\,X=L\,0=0$ and, considering $Y=R^{-1}\,X$, we have that $\,VX=0\,\Leftrightarrow\,N\,Y=0\,\Leftrightarrow\,y_1=\cdots=y_k=0$. So, the base that generates the solutions of $VX=0$ is $B^\bot=\{\,R\,e_{k+1},\dots,R\,e_4\,\}$, i.e. the set with the last $4-k$ columns of $R$.
\end{proof}

Throughout the article we will use identities among polynomials in many variables whose demonstration only requires the polynomial expansion of the difference of both members of the equalities. We will call this type of proof \emph{polynomial checking}.

\begin{prop}\label{PropInvariantFactors1}
Given a prime number $p$ and a $p-$orthonormal system $S=\{\,v_1,v_2\,\}$, $v_1=(x_1,\dots,x_4)$ and $v_2=(y_1,\dots,y_4)$, with $|\text{supp}(S)|>2$, then $\text{gcd}(x_1,\dots,x_4)=\text{gcd}(y_1,\dots,y_4)=1$ and the invariant factors of $V$ also verify $\alpha_1=\alpha_2=1$.
\end{prop}
\begin{proof}
Suppose, by contradiction, that $\text{gcd}(x_1,\dots,x_4)=g>1$. Then $N(v_1)=g^2(x_1^{\prime\,2}+\cdots+x_4^{\prime\,2})=p$, where $\displaystyle x_i^\prime=\frac{x_i}{g}$ for all $1\leq i\leq 4$, and this fact contradicts the primality of $p$. So, we have that $\text{gcd}(x_1,\dots,x_4)=1$ and in the same way we conclude that $\text{gcd}(y_1,\dots,y_4)=1$. Applying these results, together with the property of the first invariant factor, we get $\alpha_1=1$.

In order to obtain the value of $\alpha_2$ we will use the following identity, that can be proved by polynomial checking:
$$
N(v_1)N(v_2)-\langle v_1|v_2\rangle^2=
\left|
\begin{array}{cc}
x_1 & x_2 \\
y_1 & y_2
\end{array}
\right|^2+
\left|
\begin{array}{cc}
x_1 & x_3 \\
y_1 & y_3
\end{array}
\right|^2+
\cdots +
\left|
\begin{array}{cc}
x_3 & x_4 \\
y_3 & y_4
\end{array}
\right|^2
$$
By hypothesis, $N(v_1)N(v_2)-\langle v_1|v_2\rangle^2=p^2$. Suppose, again by contradiction, that $g=\text{gcd}(m_{12},\dots,m_{34})>1$, where
$$
m_{ij}=\left|
\begin{array}{cc}
x_i & x_j \\
y_i & y_j
\end{array}
\right|\qquad\text{and}\qquad
m_{ij}^\prime=\frac{m_{ij}}{g}
$$
Then $p^2=g^2(m_{12}^{\prime\,2}+\cdots+m_{34}^{\prime\,2})$ and there are, at least, two minors different from 0 because $|\text{supp}(S)|>2$. These facts contradict the primality of $p$. So, we have that $\text{gcd}(m_{12},\dots,m_{34})=1$ and, since this value matches the second invariant factor, we get $\alpha_2=1$.
\end{proof}

Finally, we introduce the fundamental result of the branch of number theory called the geometry of numbers, proved by Minkowski in 1889.

\begin{thm}[Minkowski~\cite{Ca}]\label{ThmMinkowski}
Let $K$ be a convex set in $\R^n$ which is symmetric with respect to the origin. If the volume of $K$ is greater than $2^n$ times the volume of the fundamental domain (parallelepiped) of a lattice $\Lambda$, then $K$ contains a non-zero lattice point.
\end{thm}

\section{Extended Lagrange's four-square theorem}

We are dealing with the following problem: given a prime number $p$ and a $p-$orthonormal system $S=\{\,v_1,\dots,v_k\,\}$, $1\leq k\leq 3$, with associated lattice $\Lambda$, prove that there exists $v_{k+1}\in\Lambda^\bot$ with norm $N(v_{k+1})=p$.

\begin{rem}\label{RemThmUnVector}
If the $p-$orthonormal system $S$ has a single vector $v_1=(x_1,x_2,x_3,x_4)$, the solution (valid for all $p\geq 1$) is trivial: $v_2=(x_2,-x_1,x_4,-x_3)$.
\end{rem}

\begin{rem}\label{RemThmSuppDos}
If the $p-$orthonormal system $S$ has two vectors and $|\text{supp}(S)|=2$, the solution (also valid for all $p\geq 1$) is as well trivial. Suppose, without loss of generality, that $\text{supp}(S)=\{1,2\}$ and that $v_1=(x_1,x_2,0,0)$. Then, the required vector is, for example, $v_3=(0,0,x_1,x_2)$.
\end{rem}

\subsection{Three vectors $p-$orthonormal systems}

\medskip
If the $p-$orthonormal system has three vectors, their exterior product allows us to obtain the required vector.

\begin{prop}\label{PropThmTresVectores}
Given a number $p\geq 1$ and a $p-$orthonormal system $S=\{\,v_1,v_2,v_3\,\}$, with associated lattice $\Lambda$, there exists $v_4\in\Lambda^\bot$ such that $N(v_4)=p$.
\end{prop}
\begin{proof}
Given the coordinates of the three vectors of $S$, $v_1=(x_1,x_2,x_3,x_4)$, $v_2=(y_1,y_2,$ $y_3,y_4)$ and $v_3=(z_1,z_2,z_3,z_4)$, we consider the exterior product $t=(t_1,t_2,t_3,t_4)$ where
$$
t_1=-\left|
\begin{array}{ccc}
x_2 & x_3 & x_4 \\
y_2 & y_3 & y_4 \\
z_2 & z_3 & z_4
\end{array}
\right|\quad\cdots\quad
t_4=\left|
\begin{array}{ccc}
x_1 & x_2 & x_3 \\
y_1 & y_2 & y_3 \\
z_1 & z_2 & z_3
\end{array}
\right|
$$
It can be proved that $t\in\Lambda^\bot$, by polynomial checking of $\langle v_i|t\rangle=0$, $1\leq i\leq 3$, and that $t_i^2=p^2(p-x_i^2-y_i^2-z_i^2)$, $1\leq i\leq 4$. In order to check the last equality, for example for $i=4$, it is enough to verify, by polynomial checking, that
$$
t_4^2=N(x)\,N(y)\,N(z)+2\langle x|y\rangle\langle x|z\rangle\langle y|z\rangle-N(x)\langle y|z\rangle^2-N(y)\langle x|z\rangle^2-N(z)\langle x|y\rangle^2,
$$
where $x=(x_1,x_2,x_3)$, $y=(y_1,y_2,y_3)$ and $z=(z_1,z_2,z_3)$, to replace the following values
$$
\begin{array}{ll}
N(x)=p-x_4^2 & \qquad\langle x|y\rangle=-x_4y_4 \\
N(y)=p-y_4^2 & \qquad\langle x|z\rangle=-x_4z_4 \\
N(z)=p-z_4^2 & \qquad\langle y|z\rangle=-y_4z_4
\end{array}
$$
and to test the expression obtained by replacing $t_4^2$ with $p^2(p-x_4^2-y_4^2-z_4^2)$ by polynomial checking. Finally, $\displaystyle v_4=\frac{t}{p}$ has the required properties: $v_4\in\Lambda^\bot$ and $N(v_4)=p$.
\end{proof}

\subsection{A two vectors $p-$orthonormal system $S$ with $|\text{supp}(S)|>2$}

\medskip
First of all, let us get a base of $\Lambda^\bot$, $B^\bot$, by computing a Smith quasi-normal form in which $L\in GL_k(\Q)$. Note that in this case lemma~\ref{LemaBaseLatticeOrtogonal} also holds. Let $V$ be the coordinate matrix of the $p-$orthonormal system $S=\{\,v_1,v_2\,\}$ with $|\text{supp}(S)|>2$, $v_1=(x_1,x_2,x_3,x_4)$, $v_2=(y_1,y_2,y_3,y_4)$ and $p\geq 1$. Suppose, rearranging the coordinates of $v_1$ and $v_2$ if necessary, that
$$
x_1\neq 0,\quad
\left|
\begin{array}{cc}
x_1 & x_2 \\
y_1 & y_2
\end{array}
\right| \neq 0\quad \text{and} \quad
4\in\text{supp}(S),\ \text{i.e.}\ x_4\neq 0\ \text{or}\ y_4\neq 0
$$
The Smith quasi-normal form of $S$ is:
$$
L\,VR =
\left(\begin{array}{cccc}
c & 0  & 0 & 0 \\
0 & cd & 0 & 0
\end{array}\right)
\quad\text{such that}\quad
\begin{array}{l}
L\in GL_k(\Q) \\
R\in GL_4(\Z) \\
0<c,d \\
R=R_1\, R_2\, R_3\, R_4\, R_5 \\
\end{array}
$$
where the matrices $L$ and $R_i$, $1\leq i\leq 5$, and the parameters $c$ and $d$ are those that appear in table~\ref{t:SQNF}.

\begin{table}[ht]
\centering
\begin{tabular}{|l|}
\hline
\footnotesize $\displaystyle R_1=\left(\begin{array}{cccc} \sigma_1 & \frac{-x_2}{c_1} & 0 & 0 \\ \tau_1 & \frac{x_1}{c_1} & 0 & 0 \\ 0 & 0 & 1 & 0 \\ 0 & 0 & 0 & 1 \end{array}\right)\vline height 1.1cm depth 0.9cm width 0.0cm
\begin{array}{l}
x_1\sigma_1+x_2\tau_1=c_1=\text{gcd}(x_1,x_2) \\
y_1^{\prime}=\sigma_1y_1+\tau_1y_2 \\
\displaystyle y_2^{\prime}=\frac{-x_2}{c_1}y_1+\frac{x_1}{c_1}y_2
\end{array}$ \\
\hline
\footnotesize $\displaystyle R_2=\left(\begin{array}{cccc} \sigma_2 & 0 & \frac{-x_3}{c_2} & 0 \\ 0 & 1 & 0 & 0 \\ \tau_2 & 0 & \frac{c_1}{c_2} & 0 \\ 0 & 0 & 0 & 1 \end{array}\right)\vline height 1.1cm depth 0.9cm width 0.0cm
\begin{array}{l}
c_1\sigma_2+x_3\tau_2=c_2=\text{gcd}(c_1,x_3) \\
y_1^{\prime\prime}=\sigma_2y_1^{\prime}+\tau_2y_3=\sigma_2\sigma_1y_1+\sigma_2\tau_1y_2+\tau_2y_3 \\
\displaystyle y_3^{\prime}=\frac{-x_3}{c_2}y_1^{\prime}+\frac{c_1}{c_2}y_3=\frac{-x_3}{c_2}\sigma_1y_1+\frac{-x_3}{c_2}\tau_1y_2 + \frac{c_1}{c_2}y_3
\end{array}$ \\
\hline
\footnotesize $\displaystyle R_3=\left(\begin{array}{cccc} \sigma_3 & 0 & 0 & \frac{-x_4}{c} \\ 0 & 1 & 0 & 0 \\ 0 & 0 & 1 & 0 \\ \tau_3 & 0 & 0 & \frac{c_2}{c}\end{array}\right)\vline height 1.1cm depth 0.9cm width 0.0cm
\begin{array}{l}
c_2\sigma_3+x_4\tau_3=c=\text{gcd}(c_2,x_4) \\
y_1^{\prime\prime\prime}=\sigma_3y_1^{\prime\prime}+\tau_3y_4=\sigma_3\sigma_2\sigma_1y_1+\sigma_3\sigma_2\tau_1y_2+\sigma_3\tau_2y_3+\tau_3y_4 \\
\displaystyle y_4^{\prime}=\frac{-x_4}{c}y_1^{\prime\prime}+\frac{c_2}{c}y_4=\frac{-x_4}{c}\sigma_2\sigma_1y_1 + \frac{-x_4}{c}\sigma_2\tau_1y_2+\frac{-x_4}{c}\tau_2y_3+\frac{c_2}{c}y_4
\end{array}$ \\
\hline
\footnotesize $\displaystyle L=\left(\begin{array}{cc} 1 & 0 \\ -y_1^{\prime\prime\prime} & c \end{array}\right)$\vline height 0.7cm depth 0.5cm width 0.0cm \\
\hline
\footnotesize $\displaystyle R_4=\left(\begin{array}{cccc} 1 & 0 & 0 & 0 \\ 0 & \sigma_4 & \frac{-y_3^{\prime}}{d_1} & 0 \\ 0 & \tau_4 & \frac{y_2^{\prime}}{d_1} & 0 \\ 0 & 0 & 0 & 1 \end{array}\right)\vline height 1.2cm depth 1.0cm width 0.0cm
\begin{array}{l}
y_2^{\prime}\sigma_4+y_3^{\prime}\tau_4=d_1=\text{gcd}(y_2^{\prime},y_3^{\prime})
\end{array}$ \\
\hline
\footnotesize $\displaystyle R_5=\left(\begin{array}{cccc} 1 & 0 & 0 & 0 \\ 0 & \sigma_5 & 0 & \frac{-y_4^{\prime}}{d} \\ 0 & 0 & 1 & 0 \\ 0 & \tau_5 & 0 & \frac{d_1}{d} \end{array}\right)\vline height 1.1cm depth 0.9cm width 0.0cm
\begin{array}{l}
d_1\sigma_5+y_4^{\prime}\tau_5=d=\text{gcd}(d_1,y_4^{\prime})
\end{array}$ \\
\hline
\end{tabular}
\caption{Smith quasi-normal form data.}
\label{t:SQNF}
\end{table}

\begin{lem}\label{LemBaseOrtogonal}
Given a number $p\geq 1$ and a $p-$orthonormal system $S=\{\,v_1,v_2\,\}$ with associated lattice $\Lambda$, then $B^\bot=\{\,w_1,w_2\,\}$ is a base of $\Lambda^\bot$, where
$$
\begin{array}{lll}
\vline height 0.7cm depth 0.6cm width 0pt w_1 & = & \displaystyle\left(\frac{x_2\, y_3^{\prime}}{c_1\, d_1}-\frac{x_3\, y_2^{\prime}\,\sigma_1}{c_2\, d_1},
-\frac{x_1\, y_3^{\prime}}{c_1\, d_1}-\frac{x_3\, y_2^{\prime}\,\tau_1}{c_2\, d_1},
\frac{c_1\, y_2^{\prime}}{c_2\, d_1},0\right) \\
\vline height 0.7cm depth 0.6cm width 0pt w_2 & = & \displaystyle\left(\frac{y_4^{\prime}(c_1\, x_3\,\sigma_1\,\tau_4+c_2\, x_2\,\sigma_4)}{c_1\, c_2\, d} - \frac{d_1\, x_4\, \sigma_1\,\sigma_2}{c\, d},\right. \\
\vline height 0.7cm depth 0.6cm width 0pt  & & \displaystyle\quad \left. \frac{y_4^{\prime}(c_1\, x_3\,\tau_1\,\tau_4-c_2\, x_1\,\sigma_4)}{c_1\, c_2\, d} - \frac{d_1\, x_4\, \sigma_2\,\tau_1}{c\, d}, -\frac{d_1\, x_4\,\tau_2}{c\, d}-\frac{c_1\, y_4^{\prime}\,\tau_4}{c_2\, d},
 \frac{c_2\, d_1}{c\, d}\right) \\
\end{array}
$$
\end{lem}
\begin{proof}
We obtain the result just by multiplying the matrices $R_1$, $R_2$, $R_3$, $R_4$ and $R_5$ and applying lemma~\ref{LemaBaseLatticeOrtogonal} to the Smith quasi-normal form of $S$.
\end{proof}

\begin{rem}\label{RemDetLattice}
Let $V$ and $G_V$ be the coordinate matrix and the Gram matrix, respectively, of the set of vectors $B\cup B^\bot$ and let $G$ be the Gram matrix of the set of vectors $B^\bot$. Then, $\text{det}^2(V)=\text{det}(G_V)=p^2\text{det}(G)$ and, since $\text{det}^2(\Lambda^\bot)=\text{det}(G)$, we concluded that $\displaystyle\text{det}(\Lambda^\bot)=\frac{\left|\text{det}(V)\right|}{p}$.
\end{rem}

We can use remark~\ref{RemDetLattice} to compute $\text{det}(\Lambda^\bot)$ and, indirectly, to study the matrix $G$, considered as a symmetric positive definite quadratic form.

\begin{prop}\label{PropDetLattice}
Given a number $p\geq 1$ and a $p-$orthonormal system $S=\{\,v_1,v_2\,\}$, with associated lattice $\Lambda$, then
 $\displaystyle\text{det}(\Lambda^\bot)=\frac{p}{c\,d}$, where $c$ and $d$ are the parameters that appear in table~\ref{t:SQNF}.
\end{prop}
\begin{proof}
To obtain the result we only have to compute $\text{det(V)}$, by remark~\ref{RemDetLattice}. Developing the expression of the determinant of $V$,
where $w_1$ and $w_2$ are the vectors obtained in lemma~\ref{LemBaseOrtogonal}, we obtain:
$$\displaystyle
\begin{array}{ccl}
\text{det}(V) c_1 c_2 d_1 cd & = &
c y_4^{\prime}(c_1(x_1^2 y_4-x_1 x_4 y_1+x_2(x_2 y_4-x_4 y_2)) + \\
& & \quad\ \ \, x_3(\underline{x_1\sigma_1+x_2\tau_1})(x_3 y_4-x_4 y_3))(\underline{y_2^{\prime}\sigma_4+y_3^{\prime}\tau_4}) + \\
& & d_1(c_1^2 y_2^{\prime}(c_2(x_1 y_2-x_2 y_1)+x_1 x_4 y_4\sigma_2\tau_1 - \\
& & \quad\quad\quad\ \, x_4\sigma_2(x_2 y_4\sigma_1+x_4(y_1\tau_1-y_2\sigma_1))) + \\
& & \quad\ c_1 x_3 y_2^{\prime}(c_2(x_1 y_3\tau_1-x_2 y_3\sigma_1 + x_3(y_2\sigma_1-y_1\tau_1)) + \\
& & \quad\quad\quad\quad \ x_4\tau_2(x_1 y_4\tau_1-x_2 y_4\sigma_1+x_4(y_2\sigma_1-y_1\tau_1))) + \\
& & \quad\ c_2 y_3^{\prime}(c_2(x_1^2 y_3-x_1 x_3 y_1+x_2(x_2 y_3-x_3 y_2)) + \\
& & \quad\quad\quad\ \, x_4(x_1^2 y_4\tau_2-x_1(x_3 y_4\sigma_1\sigma_2+x_4(y_1\tau_2-y_3\sigma_1\sigma_2)) + \\
& & \quad\quad\quad\quad\ \ \, x_2(x_2 y_4\tau_2-x_3 y_4\sigma_2\tau_1+x_4(y_3\sigma_2\tau_1-y_2\tau_2)))))
\end{array}
$$
where all the parameters appear in table~\ref{t:SQNF}.

Throughout the proof we will replace expressions by applying equalities from table~\ref{t:SQNF}.

Substituting the underlined expressions by $c_1$ and $d_1$ respectively, all occurrences of $d_1$ are canceled. Similarly, substituting $c_1y_2^{\prime}$, $c_2y_3^{\prime}$ and $cy_4^{\prime}$ for the expressions
$$
\begin{array}{l}
x_1y_2-x_2y_1, \\
c_1y_3-x_3(\sigma_1y_1+\tau_1y_2)\ \text{and} \\
c_2y_4-x_4(\sigma_2\sigma_1y_1+\sigma_2\tau_1y_2+\tau_2y_3)
\end{array}
$$
respectively, the parameter $c$ disappears from the second equality member.

\begin{table}[t]
\centering
\begin{tabular}{|c|c|l|r|c|c|l|r|c|c|l|}
\cline{1-3} \cline{5-7} \cline{9-11}
\footnotesize $1$ & & \footnotesize $c_1 c_2 x_1^2 y_2^2$ & \tiny & \footnotesize $2$ & & \footnotesize $c_1 c_2 x_1^2 y_3^2$ & \tiny & \footnotesize $3$ & & \footnotesize $c_1 c_2 x_1^2 y_4^2$ \\
\cline{1-3} \cline{5-7} \cline{9-11}
\footnotesize $4$ & & \footnotesize $- 2 c_1 c_2 x_1 x_2 y_1 y_2$ & \tiny & \footnotesize $5$ & \footnotesize $\times$ & \footnotesize $- c_1 c_2 x_1 x_3 y_1 y_3$ & \tiny & \footnotesize $6$ & \footnotesize $\times$ & \footnotesize $- c_1 c_2 x_1 x_4 y_1 y_4$ \\
\cline{1-3} \cline{5-7} \cline{9-11}
\footnotesize $7$ & & \footnotesize $c_1 c_2 x_2^2 y_1^2$ & \tiny & \footnotesize $8$ & & \footnotesize $c_1 c_2 x_2^2 y_3^2$ & \tiny & \footnotesize $9$ & & \footnotesize $c_1 c_2 x_2^2 y_4^2$ \\
\cline{1-3} \cline{5-7} \cline{9-11}
\footnotesize $10$ & \footnotesize $\times$ & \footnotesize $- c_1 c_2 x_2 x_3 y_2 y_3$ & \tiny & \footnotesize $11$ & \footnotesize $\times$ & \footnotesize $- c_1 c_2 x_2 x_4 y_2 y_4$ & \tiny & \footnotesize $12$ & & \footnotesize $c_1 c_2 x_3^2 y_4^2$ \\
\cline{1-3} \cline{5-7} \cline{9-11}
\footnotesize $13$ & \footnotesize $\times$ & \footnotesize $- c_1 c_2 x_3 x_4 y_3 y_4$ & \tiny & \footnotesize $14$ & \footnotesize $\times$ & \footnotesize $- c_1 x_1^2 x_4 y_1 y_4 \sigma_1 \sigma_2$ & \tiny & \footnotesize $15$ & \footnotesize $\times$ & \footnotesize $- c_1 x_1 x_2 x_4 y_1 y_4 \sigma_2 \tau_1$ \\
\cline{1-3} \cline{5-7} \cline{9-11}
\footnotesize $16$ & \footnotesize $\times$ & \footnotesize $- c_1 x_1 x_2 x_4 y_2 y_4 \sigma_1 \sigma_2$ & \tiny & \footnotesize $17$ & \footnotesize $\times$ & \footnotesize $- c_1 x_1 x_3 x_4 y_3 y_4 \sigma_1 \sigma_2$ & \tiny & \footnotesize $18$ & \footnotesize $\times$ & \footnotesize $c_1 x_1 x_4^2 y_1^2 \sigma_1 \sigma_2$ \\
\cline{1-3} \cline{5-7} \cline{9-11}
\footnotesize $19$ & \footnotesize $\times$ & \footnotesize $c_1 x_1 x_4^2 y_2^2 \sigma_1 \sigma_2$ & \tiny & \footnotesize $20$ & \footnotesize $\times$ & \footnotesize $c_1 x_1 x_4^2 y_3^2 \sigma_1 \sigma_2$ & \tiny & \footnotesize $21$ & \footnotesize $\times$ & \footnotesize $- c_1 x_2^2 x_4 y_2 y_4 \sigma_2 \tau_1$ \\
\cline{1-3} \cline{5-7} \cline{9-11}
\footnotesize $22$ & \footnotesize $\times$ & \footnotesize $- c_1 x_2 x_3 x_4 y_3 y_4 \sigma_2 \tau_1$ & \tiny & \footnotesize $23$ & \footnotesize $\times$ & \footnotesize $c_1 x_2 x_4^2 y_1^2 \sigma_2 \tau_1$ & \tiny & \footnotesize $24$ & \footnotesize $\times$ & \footnotesize $c_1 x_2 x_4^2 y_2^2 \sigma_2 \tau_1$ \\
\cline{1-3} \cline{5-7} \cline{9-11}
\footnotesize $25$ & \footnotesize $\times$ & \footnotesize $c_1 x_2 x_4^2 y_3^2 \sigma_2 \tau_1$ & \tiny & \footnotesize $26$ & \footnotesize $\times$ & \footnotesize $- c_1 x_3^2 x_4 y_1 y_4 \sigma_1 \sigma_2$ & \tiny & \footnotesize $27$ & \footnotesize $\times$ & \footnotesize $- c_1 x_3^2 x_4 y_2 y_4 \sigma_2 \tau_1$ \\
\cline{1-3} \cline{5-7} \cline{9-11}
\footnotesize $28$ & \footnotesize $\times$ & \footnotesize $- c_1 x_3^2 x_4 y_3 y_4 \tau_2$ & \tiny & \footnotesize $29$ & \footnotesize $\times$ & \footnotesize $c_1 x_3 x_4^2 y_1 y_3 \sigma_1 \sigma_2$ & \tiny & \footnotesize $30$ & \footnotesize $\times$ & \footnotesize $c_1 x_3 x_4^2 y_2 y_3 \sigma_2 \tau_1$ \\
\cline{1-3} \cline{5-7} \cline{9-11}
\footnotesize $31$ & \footnotesize $\times$ & \footnotesize $c_1 x_3 x_4^2 y_3^2 \tau_2$ & \tiny & \footnotesize $32$ & \footnotesize $\times$ & \footnotesize $- c_2 x_1^2 x_3 y_1 y_3 \sigma_1$ & \tiny & \footnotesize $33$ & \footnotesize $\times$ & \footnotesize $- c_2 x_1 x_2 x_3 y_1 y_3 \tau_1$ \\
\cline{1-3} \cline{5-7} \cline{9-11}
\footnotesize $34$ & \footnotesize $\times$ & \footnotesize $- c_2 x_1 x_2 x_3 y_2 y_3 \sigma_1$ & \tiny & \footnotesize $35$ & \footnotesize $\times$ & \footnotesize $c_2 x_1 x_3^2 y_1^2 \sigma_1$ & \tiny & \footnotesize $36$ & \footnotesize $\times$ & \footnotesize $c_2 x_1 x_3^2 y_2^2 \sigma_1$ \\
\cline{1-3} \cline{5-7} \cline{9-11}
\footnotesize $37$ & \footnotesize $\times$ & \footnotesize $- c_2 x_2^2 x_3 y_2 y_3 \tau_1$ & \tiny & \footnotesize $38$ & \footnotesize $\times$ & \footnotesize $c_2 x_2 x_3^2 y_1^2 \tau_1$ & \tiny & \footnotesize $39$ & \footnotesize $\times$ & \footnotesize $c_2 x_2 x_3^2 y_2^2 \tau_1$\\
\cline{1-3} \cline{5-7} \cline{9-11}
\footnotesize $40$ & \footnotesize $\times$ & \footnotesize $- x_1^2 x_3 x_4 y_1 y_4 \sigma_1 \tau_2$ & \tiny & \footnotesize $41$ & \footnotesize $\times$ & \footnotesize $- x_1 x_2 x_3 x_4 y_1 y_4 \tau_1 \tau_2$ & \tiny & \footnotesize $42$ & \footnotesize $\times$ & \footnotesize $- x_1 x_2 x_3 x_4 y_2 y_4 \sigma_1 \tau_2$ \\
\cline{1-3} \cline{5-7} \cline{9-11}
\footnotesize $43$ & \footnotesize $\times$ & \footnotesize $x_1 x_3^2 x_4 y_1 y_4 \sigma_1^2 \sigma_2$ & \tiny & \footnotesize $44$ & \footnotesize $\times$ & \footnotesize $x_1 x_3^2 x_4 y_2 y_4 \sigma_1 \sigma_2 \tau_1$ & \tiny & \footnotesize $45$ & \footnotesize $\times$ & \footnotesize $x_1 x_3 x_4^2 y_1^2 \sigma_1 \tau_2$ \\
\cline{1-3} \cline{5-7} \cline{9-11}
\footnotesize $46$ & \footnotesize $\times$ & \footnotesize $- x_1 x_3 x_4^2 y_1 y_3 \sigma_1^2 \sigma_2$ & \tiny & \footnotesize $47$ & \footnotesize $\times$ & \footnotesize $x_1 x_3 x_4^2 y_2^2 \sigma_1 \tau_2$ & \tiny & \footnotesize $48$ & \footnotesize $\times$ & \footnotesize $- x_1 x_3 x_4^2 y_2 y_3 \sigma_1 \sigma_2 \tau_1$ \\
\cline{1-3} \cline{5-7} \cline{9-11}
\footnotesize $49$ & \footnotesize $\times$ & \footnotesize $- x_2^2 x_3 x_4 y_2 y_4 \tau_1 \tau_2$ & \tiny & \footnotesize $50$ & \footnotesize $\times$ & \footnotesize $x_2 x_3^2 x_4 y_1 y_4 \sigma_1 \sigma_2 \tau_1$ & \tiny & \footnotesize $51$ & \footnotesize $\times$ & \footnotesize $x_2 x_3^2 x_4 y_2 y_4 \sigma_2 \tau_1^2$ \\
\cline{1-3} \cline{5-7} \cline{9-11}
\footnotesize $52$ & \footnotesize $\times$ & \footnotesize $x_2 x_3 x_4^2 y_1^2 \tau_1 \tau_2$ & \tiny & \footnotesize $53$ & \footnotesize $\times$ & \footnotesize $- x_2 x_3 x_4^2 y_1 y_3 \sigma_1 \sigma_2 \tau_1$ & \tiny & \footnotesize $54$ & \footnotesize $\times$ & \footnotesize $x_2 x_3 x_4^2 y_2^2 \tau_1 \tau_2$ \\
\cline{1-3} \cline{5-7} \cline{9-11}
\footnotesize $55$ & \footnotesize $\times$ & \footnotesize $- x_2 x_3 x_4^2 y_2 y_3 \sigma_2 \tau_1^2$ & \tiny & & & & \tiny & & & \\
\cline{1-3} \cline{5-7} \cline{9-11}
\end{tabular}
\caption{Monomials of $\text{det}(V) c_1 c_2 cd$.}
\label{t:Monomios}
\end{table}

\begin{table}[t]
\centering
\begin{tabular}{|c|c|c|l|r|c|c|c|l|}
\cline{1-4} \cline{6-9}
\footnotesize $14$ & \footnotesize $\times$ & \footnotesize $15$ & \footnotesize $- c_1^2 x_1 x_4 y_1 y_4 \sigma_2$ & & \footnotesize $16$ & \footnotesize $\times$ & \footnotesize $21$ & \footnotesize $- c_1^2 x_2 x_4 y_2 y_4 \sigma_2$ \\
\cline{1-4} \cline{6-9}
\footnotesize $17$ & \footnotesize $\times$ & \footnotesize
 $22$ & \footnotesize $- c_1^2 x_3 x_4 y_3 y_4 \sigma_2$ & & \footnotesize $18$ & \footnotesize $\times$ & \footnotesize $23$ & \footnotesize $c_1^2 x_4^2 y_1^2 \sigma_2$ \\
\cline{1-4} \cline{6-9}
\footnotesize $19$ & \footnotesize $\times$ & \footnotesize $24$ & \footnotesize $c_1^2 x_4^2 y_2^2 \sigma_2$ & & \footnotesize $20$ & \footnotesize $\times$ & \footnotesize $25$ & \footnotesize $c_1^2 x_4^2 y_3^2 \sigma_2$ \\
\cline{1-4} \cline{6-9}
\footnotesize $32$ & \footnotesize $\times$ & \footnotesize $33$ & \footnotesize $- c_1 c_2 x_1 x_3 y_1 y_3$ & & \footnotesize $34$ & \footnotesize $\times$ & \footnotesize $37$ & \footnotesize $- c_1 c_2 x_2 x_3 y_2 y_3$ \\
\cline{1-4} \cline{6-9}
\footnotesize $35$ & & \footnotesize $38$ & \footnotesize $c_1 c_2 x_3^2 y_1^2$ & & \footnotesize $36$ & & \footnotesize $39$ & \footnotesize $c_1 c_2 x_3^2 y_2^2$ \\
\cline{1-4} \cline{6-9}
\footnotesize $40$ & \footnotesize $\times$ & \footnotesize $41$ & \footnotesize $- c_1 x_1 x_3 x_4 y_1 y_4 \tau_2$ & & \footnotesize $42$ & \footnotesize $\times$ & \footnotesize $49$ & \footnotesize $- c_1 x_2 x_3 x_4 y_2 y_4 \tau_2$ \\
\cline{1-4} \cline{6-9}
\footnotesize $43$ & \footnotesize $\times$ & \footnotesize $50$ & \footnotesize $c_1 x_3^2 x_4 y_1 y_4 \sigma_1 \sigma_2$ & & \footnotesize $44$ & \footnotesize $\times$ & \footnotesize $51$ & \footnotesize $c_1 x_3^2 x_4 y_2 y_4 \sigma_2 \tau_1$ \\
\cline{1-4} \cline{6-9}
\footnotesize $45$ & \footnotesize $\times$ & \footnotesize $52$ & \footnotesize $c_1 x_3 x_4^2 y_1^2 \tau_2$ & & \footnotesize $46$ & \footnotesize $\times$ & \footnotesize $53$ & \footnotesize $- c_1 x_3 x_4^2 y_1 y_3 \sigma_1 \sigma_2$ \\
\cline{1-4} \cline{6-9}
\footnotesize $47$ & \footnotesize $\times$ & \footnotesize $54$ & \footnotesize $c_1 x_3 x_4^2 y_2^2 \tau_2$ & & \footnotesize $48$ & \footnotesize $\times$ & \footnotesize $55$ & \footnotesize $- c_1 x_3 x_4^2 y_2 y_3 \sigma_2 \tau_1$ \\
\hline
\hline
\footnotesize $14$ & \footnotesize $\times$ & \footnotesize $40$ & \footnotesize $- c_1 c_2 x_1 x_4 y_1 y_4$ & & \footnotesize $16$ & \footnotesize $\times$ & \footnotesize $42$ & \footnotesize $- c_1 c_2 x_2 x_4 y_2 y_4$ \\
\cline{1-4} \cline{6-9}
\footnotesize $17$ & \footnotesize $\times$ & \footnotesize $28$ & \footnotesize $- c_1 c_2 x_3 x_4 y_3 y_4$ & & \footnotesize $18$ & & \footnotesize $45$ & \footnotesize $c_1 c_2 x_4^2 y_1^2$ \\
\cline{1-4} \cline{6-9}
\footnotesize $19$ & & \footnotesize $47$ & \footnotesize $c_1 c_2 x_4^2 y_2^2$ & & \footnotesize $20$ & & \footnotesize $31$ & \footnotesize $c_1 c_2 x_4^2 y_3^2$ \\
\cline{1-4} \cline{6-9}
\footnotesize $26$ & \footnotesize $\times$ & \footnotesize $43$ & \footnotesize $0$ & & \footnotesize $27$ & \footnotesize $\times$ & \footnotesize $44$ & \footnotesize $0$ \\
\cline{1-4} \cline{6-9}
\footnotesize $29$ & \footnotesize $\times$ & \footnotesize $46$ & \footnotesize $0$ & & \footnotesize $30$ & \footnotesize $\times$ & \footnotesize $48$ & \footnotesize $0$ \\
\hline
\hline
\footnotesize $5$ & & \footnotesize $32$ & \footnotesize $- 2c_1c_2x_1x_3y_1y_3$ & & \footnotesize $6$ & & \footnotesize $14$ & \footnotesize $- 2c_1c_2x_1x_4y_1y_4$ \\
\cline{1-4} \cline{6-9}
\footnotesize $10$ & & \footnotesize $34$ & \footnotesize $- 2c_1c_2x_2x_3y_2y_3$ & & \footnotesize $11$ & & \footnotesize $16$ & \footnotesize $- 2c_1c_2x_2x_4y_2y_4$ \\
\cline{1-4} \cline{6-9}
\footnotesize $13$ & & \footnotesize $17$ & \footnotesize $- 2c_1c_2x_3x_4y_3y_4$ & & & & & \\
\cline{1-4} \cline{6-9}
\end{tabular}
\caption{Monomials resulting from operations.}
\label{t:Substituciones}
\end{table}

The expression $\text{det}(V) c_1 c_2 cd$ is a homogeneous polynomial of total degree $6$ in the variables $c_1$, $c_2$, $x_1$, $x_2$, $x_3$, $x_4$, $y_1$, $y_2$, $y_3$ and $y_4$, in which only the parameters $\sigma_1$, $\tau_1$, $\sigma_2$ and $\tau_2$ appear. The monomials of the aforementioned polynomial are included in table~\ref{t:Monomios} and are identified by indexes placed in the first cells of the corresponding rows.

In order to eliminate the parameters $\sigma_1$, $\tau_1$, $\sigma_2$ and $\tau_2$, we group the monomials of the table~\ref{t:Monomios} in pairs to apply the following operations:
$$
\begin{array}{l}
\text{(1) Substitute}\ \ x_1\sigma_1+x_2\tau_1\ \ \text{by}\ \ c_1. \\
\text{(2) Substitute}\ \ c_1\sigma_2+x_3\tau_2\ \ \text{by}\ \ c_2. \\
\text{(3) Cancel opposite monomials}. \\
\text{(4) Add equal monomials}. \\
\end{array}
$$

Applied operations are detailed in table~\ref{t:Substituciones}, where the resulting monomials are identified by the indexes of the first monomials that are operated on. Each time an operation is applied, the monomials involved are marked with a $\times$ to the right of the index that identifies the monomial, so as not to use them again. The operations are done iteratively on monomials of tables~\ref{t:Monomios} and ~\ref{t:Substituciones} that are not marked, until no operation can be further applied.

All the resulting monomials have the factor $c_1c_2$. Therefore, by simplifying this factor the next equality is obtained:
$$
\begin{array}{ccl}
\vline height 0.4cm depth 0.3cm width 0pt \text{det}(V)cd & = & x_1^2 y_2^2 + x_1^2 y_3^2 + x_1^2 y_4^2 -2 x_1 x_2 y_1 y_2 - 2x_1x_3y_1y_3 - 2x_1x_4y_1y_4 \\
\vline height 0.4cm depth 0.3cm width 0pt & & x_2^2 y_1^2 + x_2^2 y_3^2 + x_2^2 y_4^2 - 2x_2x_3y_2y_3 - 2x_2x_4y_2y_4 + x_3^2 y_4^2 \\
\vline height 0.4cm depth 0.3cm width 0pt & & - 2x_3x_4y_3y_4 + x_4^2 y_1^2 + x_4^2 y_2^2 + x_4^2 y_3^2 + x_3^2 y_1^2 + x_3^2 y_2^2
\end{array}
$$

By polynomial checking, it is easy to verify the next equality:
$$
\text{det}(V)cd=(x_1^2+x_2^2+x_3^2+x_4^2)(y_1^2+y_2^2+y_3^2+y_4^2)-(x_1y_1+x_2y_2+x_3y_3+x_4y_4)^2
$$

By hypothesis, the second member of the previous equality is equal to $p^2$. Therefore, by applying remark~\ref{RemDetLattice}, we conclude that:
\begin{gather*}
\displaystyle
\text{det}(\Lambda^\bot)=\frac{p}{cd}\qedhere
\end{gather*}
\end{proof}

\begin{lem}\label{LemNormaW1}
Given a number $p\geq 1$, a $p-$orthonormal system $S=\{\,v_1,v_2\,\}$ and $w_1$ the first vector of the base $B^\bot$ of the orthogonal lattice $\Lambda^\bot$,
 then $\displaystyle N(w_1)=\frac{p(p-x_4^2-y_4^2)}{c_2^2\,d_1^2}$, where $c_2$ and $d_1$ are the parameters in table~\ref{t:SQNF}.
\end{lem}
\begin{proof}
The proof is similar to that of proposition~\ref{PropDetLattice}. Considering the vector $w_1$ obtained in lemma~\ref{LemBaseOrtogonal} and calculating $N(w_1)$, the following equality is obtained:
$$
N(w_1)c_1^2c_2^2d_1^2=c_1^4y_2^{\prime\,2} + c_1^2x_3^2y_2^{\prime\,2}(\sigma_1^2 + \tau_1^2) + 2c_1c_2x_3y_2^{\prime}y_3^{\prime}(x_1\tau_1 - x_2\sigma_1) + c_2^2y_3^{\prime\,2}(x_1^2 + x_2^2)
$$

Substituting in the second member of equality $c_1y_2^{\prime}$ by $-x_2y_1+x_1y_2$ and $c_2y_3^{\prime}$ by $-x_3\sigma_1y_1-x_3\tau_1y_2+c_1y_3$, a homogeneous polynomial of total grade $6$ in the variables $c_1$, $x_1$, $x_2$, $x_3$, $y_1$, $y_2$ and $y_3$ is obtained, in which only the parameters $\sigma_1$ and $\tau_1$ appear.

The monomials of the aforementioned polynomial are listed in table~\ref{t:MonomiosN(w)}. The results of the following substitution are also included in the table:
replace $x_1\sigma_1+x_2\tau_1$ by $c_1$.

All the remaining monomials are multiplied by the factor $c_1^2$. Therefore, simplifying this factor, we obtain:
$$
\begin{array}{ccl}
\vline height 0.4cm depth 0.3cm width 0pt N(w_1)c_2^2d_1^2 & = & x_1^2y_2^2 + x_1^2y_3^2 - 2x_1x_2y_1y_2 + x_2^2y_1^2 + x_2^2y_3^2 \\
\vline height 0.4cm depth 0.3cm width 0pt & & - 2x_1x_3y_1y_3 - 2x_2x_3y_2y_3 + x_3^2y_1^2 + x_3^2y_2^2 \\
\end{array}
$$

\begin{table}[t]
\centering
\begin{tabular}{|l|r|l|l|r|l|l|}
\cline{1-1} \cline{3-4} \cline{6-7}
\footnotesize $c_1^2x_1^2y_2^2$ & & \footnotesize $-2c_1x_1^2x_3y_1y_3\sigma_1$ & \footnotesize & & \footnotesize $-2c_1x_1x_2x_3y_2y_3\sigma_1$ & \footnotesize \\
\cline{1-1} \cline{3-3} \cline{6-6}
\footnotesize $c_1^2x_1^2y_3^2$ & & \footnotesize $-2c_1x_1x_2x_3y_1y_3\tau_1$ & \footnotesize $-2c_1^2x_1x_3y_1y_3$ & & \footnotesize $-2c_1x_2^2x_3y_2y_3\tau_1$ & \footnotesize $-2c_1^2x_2x_3y_2y_3$ \\
\cline{1-1} \cline{3-4} \cline{6-7}
\footnotesize $- 2c_1^2x_1x_2y_1y_2$ & & \footnotesize $x_1^2x_3^2y_1^2\sigma_1^2$ & \footnotesize & & \footnotesize $x_1^2x_3^2y_2^2\sigma_1^2$ & \footnotesize \\
\cline{1-1} \cline{3-3} \cline{6-6}
\footnotesize $c_1^2x_2^2y_1^2$ & & \footnotesize $x_2^2x_3^2y_1^2\tau_1^2$ & \footnotesize $c_1^2x_3^2y_1^2$ & & \footnotesize $x_2^2x_3^2y_2^2\tau_1^2$ & \footnotesize $c_1^2x_3^2y_2^2$ \\
\cline{1-1} \cline{3-3} \cline{6-6}
\footnotesize $c_1^2x_2^2y_3^2$ & & \footnotesize $2x_1x_2x_3^2y_1^2\sigma_1\tau_1$ & \footnotesize & & \footnotesize $2x_1x_2x_3^2y_2^2\sigma_1\tau_1$ & \footnotesize \\
\cline{1-1} \cline{3-4} \cline{6-7}
\end{tabular}
\caption{Monomials of $N(w_1)c_1^2c_2^2d_1^2$.}
\label{t:MonomiosN(w)}
\end{table}

By polynomial checking, it is easy to verify the next equality:
$$
\begin{array}{ccl}
\vline height 0.4cm depth 0.3cm width 0pt N(w_1)c_2^2d_1^2 & = & (x_1^2+x_2^2+x_3^2+x_4^2)(y_1^2+y_2^2+y_3^2+y_4^2) \\
\vline height 0.4cm depth 0.3cm width 0pt & & - (x_1y_1+x_2y_2+x_3y_3+x_4y_4)^2 - x_4^2(y_1^2+y_2^2+y_3^2+y_4^2) \\
\vline height 0.4cm depth 0.3cm width 0pt & & -y_4^2(x_1^2+x_2^2+x_3^2+x_4^2) + 2x_4y_4(x_1y_1+x_2y_2+x_3y_3+x_4y_4) \\
\end{array}
$$

By hypothesis, the second member of the previous equality is equal to $p^2-px_4^2-py_4^2$. Therefore, we conclude that:
\begin{gather*}
\displaystyle
N(w_1)=\frac{p(p-x_4^2-y_4^2)}{c_2^2d_1^2}\qedhere
\end{gather*}

\end{proof}

\begin{lem}\label{LemSuppInvariantFactros1}
Given a prime number $p$ and a $p-$orthonormal system $S=\{\,v_1,v_2\,\}$ with $|\text{supp}(S)|>2$, associated to the lattice $\Lambda$, then $c=d=1$, where $c$ and $d$ are the parameters that appear in table~\ref{t:SQNF}.
\end{lem}
\begin{proof}
According to table~\ref{t:SQNF} it holds that $c=\text{gcd}(x_1,x_2,x_3,x_4)$ and, by proposition~\ref{PropInvariantFactors1}, we conclude that $c=1$. This result implies that the Smith quasi-normal form described in table~\ref{t:SQNF} is actually a normal form, because in this case $L\in GL_k(\Z)$, and consequently $d$ is the second invariant factor of $V$. Considering once more proposition~\ref{PropInvariantFactors1} we conclude that $d = 1$.
\end{proof}

\begin{prop}\label{PropPdivideG}
Given a prime number $p$, a $p-$orthonormal system $S=\{\,v_1,v_2\,\}$ with $|\text{supp}(S)|>2$ and the Gram matrix $G$ of the base $B^\bot=\{\,w_1,w_2\,\}$ of the orthogonal lattice $\Lambda^\bot$, then it holds that $p\,|\,G$.
\end{prop}
\begin{proof}
Suppose that the Gram matrix
$
G=\left(
\begin{array}{cc}
\mu & \lambda \\
\lambda & \nu
\end{array}
\right)
$.

Let us consider the value of $\mu=N(w_1)$ obtained in lemma~\ref{LemNormaW1}. The prime factorization of $p(p-x_4^2-y_4^2)$ contains only one factor $p$, because $p$ is prime and $-p<p-x_4^2-y_4^2<p$ (remember that we are assuming that $x_4\neq 0$ or $y_4\neq 0$). Then, the prime factorization of $c_2^2\,d_1^2$ does not contain $p$, because the number of times it contains each prime factor is even. Consequently $c_2^2\,d_1^2\,|\,(p-x_4^2-y_4^2)$ and this implies that $p\,|\,\mu$, i.e, $\mu=p\,\mu^\prime$. Moreover, $|\mu^\prime|<p$.

Applying proposition~\ref{PropDetLattice}, lemma~\ref{LemSuppInvariantFactros1} and the property $\text{det}^2(\Lambda^\bot)=\text{det}(G)$, we get $p^2=p\,\mu^\prime\,\nu-\lambda^2$. This implies $p\,|\,\lambda^2$ and, keeping in mind that $p$ is a prime, we have that $p\,|\,\lambda$, i.e. $\lambda=p\,\lambda^\prime$.

Reconsidering the previous equality, and canceling a factor $p$, we obtain $p=\mu^\prime\,\nu-p\,\lambda^{\prime\,2}$. This implies again that $p\,|\,\mu^\prime\,\nu$ and, considering that $p$ is prime and $|\mu^\prime|<p$, we get $p\,|\,\nu$, i.e. $\nu=p\,\nu^\prime$.

We arrive to the final conclusion that
$
G=p\left(
\begin{array}{cc}
\mu^\prime & \lambda^\prime \\
\lambda^\prime & \nu^\prime
\end{array}
\right)
$,
i.e. $p\,|\,G$.
\end{proof}

\begin{thm}\label{ThmExisteV3}
Given a prime number $p$, a $p-$orthonormal system $S=\{\,v_1,v_2\,\}$ with $|\text{supp}(S)|>2$ and associated lattices $\Lambda$ and $\Lambda^\bot$, there exists $v_3\in\Lambda^\bot$ such that it verifies $N(v_3)=p$.
\end{thm}
\begin{proof}
Let $G$ be the Gram matrix of the base $B^\bot$ of the associated lattice $\Lambda^\bot$.

Proposition~\ref{PropDetLattice}, lemma~\ref{LemSuppInvariantFactros1} and property $\text{det}^2(\Lambda^\bot)=\text{det}(G)$ allow us to conclude that $\text{det}(G)=p^2$.
Applying now proposition~\ref{PropPdivideG} we obtain that $\displaystyle G^{\prime}=\frac{G}{p}$ is an unimodular matrix, i.e. $G^\prime\in GL_2(\Z)$, and that, given a vector $v_3\in\Lambda^\bot$, $N(v_3)=b^t\,G\,b=p$ if and only if $b^t\,G^\prime\,b=1$, $b$ being the coordinate vector of $v_3$ in the base $B^\bot$.

Let $K = \{\,x\in\R^2\ |\ x^t\,G^{\prime}\,x \leq 1\,\}$ and $\{u_1,u_2\}$ be an orthonormal base of eigenvectors of $G^\prime$ with eigenvalues $\lambda_1$ and $\lambda_1$ respectively. Note that $\lambda_1$ and $\lambda_2$ are real, since $G^\prime$ is symmetric, positive, because $G^\prime$ is definite positive, and verify $\lambda_1\,\lambda_2=\text{det}(G^\prime)=1$. Then $K$ is the ellipse $\lambda_1x^2+\lambda_2y^2\leq 1$, with respect to the reference system determined by $u_1$ and $u_2$, and has volume $\pi\frac{1}{\sqrt {\lambda_1}}\frac{1}{\sqrt {\lambda_2}} = \pi$.

Given a $0<\epsilon<1$, let be $E_\epsilon$ the ellipse $K$ scaled by a factor $f_\epsilon=\frac{2}{\sqrt {\pi}}+\epsilon$. The ellipse $E_\epsilon$ has volume $\pi f_\epsilon^2>\pi\frac{2^2}{\pi}=2^2$. Then, for the Theorem~\ref{ThmMinkowski}, there exists a point $b$ in the lattice $\Z^2$ (with volume of the fundamental domain $1$) such that $b\neq 0$ and $b\in E_\epsilon$. Since the set of points of $\Z^2$ that belong to any of the ellipses $E_\epsilon$ is finite, it is shown that there is a point $b$ in the lattice $\Z^2$ such that $b\neq 0$ and $b\in K$.

The point $b$ defines a vector $v_3\in\Lambda^\bot$ that verifies $0<b^t\,G^\prime\,b\leq 1$. Then, it holds $b^t\,G^\prime\,b=1$, since $b^t\,G^\prime\,b$ is integer, and, at last, is the wanted vector of $\Lambda^\bot$, because $N(v_3)=b^t\,G\,b=p$.
\end{proof}

\subsection{Extensions of $p-$orthonormal systems}

\medskip
Putting together remark~\ref{RemThmUnVector}, remark~\ref{RemThmSuppDos}, proposition~\ref{PropThmTresVectores} and theorem~\ref{ThmExisteV3}, we obtain the following theorem.

\begin{thm}\label{ThmBaseExtensible}
Given a prime number $p$ and a $p-$orthonormal system in $\Z^4$, $S$, then $S$ can be extended to a $p-$orthonormal base.
\end{thm}

\section{Generalizations and conjectures}

We have proved that every $p-$orthonormal system of vectors in $\Z^4$ can be extended to a $p-$orthonormal base if $p$ is a prime number. Besides, we have verified the result for every $1\leq p\leq 10000$. In this section, all verifications for given values of $p$ and $n$ have been made by exhaustive checking of all $p-$orthonormal systems in $\Z^n$. From the previous results we conjecture that the following result holds.

\begin{con}
\label{ConjeturaZ4}
Given an integer number $p\geq 1$ and a $p-$orthonormal system in $\Z^4$, $S$, then $S$ can be extended to a $p-$orthonormal base.
\end{con}

The most natural generalization of the problem is to consider it in any dimension $n\geq 1$, i.e. to study the problem in $\Z^n$.

\begin{prob}
\label{prob1}
Given an integer number $p\geq 1$ and a $p-$orthonormal system in $\Z^n$, $S$, ¿can $S$ be extended to a $p-$orthonormal base?
\end{prob}

This problem arose from the study of discrete quantum states \cite{GL1}, for quantum computing. Because the dimension of the vector space of these states ($m-$qubits) is $2^m$, it would be expected that the result would be fulfilled for these dimensions.

An analogous construction to that given in remark \ref{RemThmUnVector} shows the result for $n=2$. Note that if $p$ cannot be written as a sum of two squares \cite{JJ} (the prime decomposition of $p$ contains a prime congruent to $3$ mod $4$ raised to an odd power), there are no $p-$orthonormal systems in $\Z^2$. The case of dimension $4$ has already been studied and, in the case $n=8$, we have checked the result for $1\leq p\leq 36$.

To analyze the problem in other dimensions we try to find counterexamples that help us to understand in which cases the problem has a positive answer. If $p$ is not a square and there exists a $p-$orthonormal base in $\Z^n$ then there are counterexamples for $p$ in dimension $n+1$. Indeed, let $\{v_1\dots,v_n\}$ be a $p-$orthonormal base in dimension $n$. Then $\{w_1\dots,w_n\}$ is a $p-$orthonormal system in dimension $n+1$ that cannot be extended to a $p-$orthonormal base, being:
$$
w_j=(v_{j,1},\dots,v_{j,n},0)\qquad 1\leq j\leq n
$$

This construction allows us to find counterexamples for any dimension $n\neq 0\,\text{mod}\,4$, $n\neq 1$ and $n\neq 2$. Given an integer $p\geq 1$, we consider the $p-$orthonormal base in $\Z^4$ $S_1=\{v_1,v_2,v_3,v_4\}$ and the matrix $A$,
$$\begin{array}{l}
v_1=(x_1,x_2,x_3,x_4) \\
v_2=(-x_2,x_1,-x_4,x_3) \\
v_3=(-x_3,x_4,x_1,-x_2) \\
v_4=(x_4,x_3,-x_2,-x_1)
\end{array}
\quad\text{and}\quad
A=\left(
\begin{array}{rrrr}
x_1 & x_2 & x_3 & x_4 \\
-x_2 & x_1 & -x_4 & x_3 \\
-x_3 & x_4 & x_1 & -x_2 \\
x_4 & x_3 & -x_2 & -x_1
\end{array}
\right),
$$
where $p=x_1^2+x_2^2+x_3^2+x_4^2$. If $p$ can be written as a sum of two squares, $p=y_1^2+y_2^2$, we define the $p-$orthonormal base in $\Z^2$ $S_2=\{u_1,u_2\}$ and the matrix $B$,
$$\begin{array}{l}
u_1=(y_1,y_2) \\
v_2=(-y_2,y_1)
\end{array}
\qquad\text{and}\qquad
B=\left(
\begin{array}{rrrr}
y_1 & y_2 \\
-y_2 & y_1
\end{array}
\right).
$$
Then, the rows of the matrices $C_1$, $C_2$ y $C_3$ define non-extensible $p-$orthonormal systems.
\begin{enumerate}
\item[(i)] $C_1$ if $p$ is not a square, $n=1\,\text{mod}\,4$ and $n\neq 1$.
\item[(ii)] $C_2$ if $p$ cannot be written as a sum of two squares, $n=2\,\text{mod}\,4$ and $n\neq 2$.
\item[(iii)] $C_3$ if $p$ is not a square and can be written as a sum of two squares and $n=3\,\text{mod}\,4$.
\end{enumerate}
$$
C_1=\left(
\begin{array}{ccccc}
A & \cdots & 0 & 0 \\
\vdots & \ddots & \vdots & \vdots \\
0 & \cdots & A & 0 \\
\end{array}\right)\
C_2=\left(
\begin{array}{cccccc}
A & \cdots & 0 & 0 & 0 \\
\vdots & \ddots & \vdots & \vdots & \vdots \\
0 & \cdots & A & 0 & 0 \\
\end{array}\right)\
C_3=\left(
\begin{array}{cccccc}
A & \cdots & 0 & 0 & 0 \\
\vdots & \ddots & \vdots & \vdots & \vdots \\
0 & \cdots & A & 0 & 0 \\
0 & \cdots & 0 & B & 0 \\
\end{array}\right)
$$

The experimental verifications and the previous counterexamples make us think that the generalization of conjecture \ref{ConjeturaZ4} should be the following.

\begin{con}
\label{ConjeturaZn}
Given numbers $n=0\,\text{mod}\,4$ ($n\geq 1$) and $p\geq 1$ and a $p-$orthonormal system in $\Z^n$, $S$, then $S$ can be extended to a $p-$orthonormal base.
\end{con}

But, what happens if $p$ is a square? We have verified the result for $n=3,5$ and $1^2\leq p\leq 100^2$, $n=6$ and $1^2\leq p\leq 33^2$, $n=7$ and $1^2\leq p\leq 13^2$ and $n=9$ and $1^2\leq p\leq 2^2$. Nevertheless, we have found that the problem \ref{prob1} has a negative answer if $n=9$, $p=9$ and $S=\{(1,\dots,1)\}$. This counterexample can be generalized as follows: if $n={\bar n}^2$ and $p=n{\bar p}^2$ are odd integers, then the set $S=\{v_1=(\bar p,\dots,\bar p)\}$ cannot be extended to a $p-$orthonormal base in $\Z^n$. Indeed, $S$ cannot be extended with a vector $v_2$ because, on one hand, the number of odd components of $v_2$ must be odd because $N(v_2)=p$ is odd and, on the other hand, the number of odd components of $v_2$ must be even because $\langle v_1|v_2\rangle=0$ is even. Hence, if $p$ is a square, our conjecture is as follows.

\begin{con}
Given numbers $n\geq 1$ and $p\geq 1$, so that either $n$ is even or $p$ is even or $n\nmid p$, and a $p^2-$orthonormal system in $\Z^n$, $S$, then $S$ can be extended to a $p-$orthonormal base.
\end{con}

\subsection{Structural properties of the problem}

Given the integer number $k$ and the vectors $u=(x_1,\dots,x_n)$ and $v=(y_1,\dots,y_n)$ belonging to $\Z^n$, we denote the \emph{parity of $k$} by $P(k)=k\,\text{mod}\,2$, the \emph{parity of $u$} by $P(u)=(x_1+\dots+x_n)\,\text{mod}\,2$ and the \emph{parity of $u$ and $v$} by $P(u,v)=\langle u|v\rangle\,\text{mod}\,2$. Note that $P(u)=P(N(u))$.

These definitions allow us to consider the conditions of $p-$orthonormality in terms of parities (module $2$), proving the following result.

\begin{prop}
Given a $p-$orthonormal system in $\Z^n$, $S=\{v_1,\dots,v_k\}$, then it holds that $P(p)=P(v_j)$, $1\leq j\leq k$, and $P(v_h,v_j)=0$, $1\leq h,j\leq k$.
\end{prop}

\subsection{Orthogonal extensions}

Given a set of vectors belonging to $\Z^n$, $S=\{v_1,\dots,v_k\}$, such that $\langle v_i|v_j\rangle=0$ for all $1\leq i<j\leq k$, we will say that $S$ is an \emph{orthogonal system} and, if $k=n$, that $S$ is an \emph{orthogonal base}.

The relaxation of the condition from $p-$orthonormality to orthogonality allows to extend any orthogonal system. Indeed, lemma \ref{LemaBaseLatticeOrtogonal} does not depend on the normalization of the vectors and can be applied in $\Z^n$, proving the following proposition.

\begin{prop}
\label{propOrtogonal}
Given an orthogonal system in $\Z^n$, $S$, then $S$ can be extended to an orthogonal base.
\end{prop}

Given an orthogonal set in $\Z^n$, $S=\{v_1,\dots,v_k\}$ ($1\leq k\leq n$), we denote the \emph{norm of $S$} by $N(S)=\text{max}\{N(v_j)\ |\ 1\leq j\leq k\}$. So, an interesting problem, in view of proposition \ref{propOrtogonal}, is the following:

\begin{prob}
\label{prob2}
Given an orthogonal system in $\Z^n$, $S$, determine the orthogonal base with the smaller norm that extends $S$.
\end{prob}

\end{document}